\documentclass[12pt]{amsart}
\usepackage{amsmath,amscd,amsthm,amsfonts, amssymb,amsxtra, mathrsfs, amsrefs}
\usepackage[us,12hr]{datetime}
\usepackage[all]{xy} \SelectTips{cm}{}
\usepackage{hyperref}
  \hypersetup{colorlinks=true,citecolor=blue}
   \usepackage{graphicx}
   \usepackage[font=scriptsize]{caption}
   
\usepackage{todonotes}
\usepackage{tikz-cd}

\CDat
\newtheorem{thm}{Theorem}[section]  
\newtheorem*{un-no-thm}{Theorem}
\newtheorem{cor}[thm]{Corollary}     
\newtheorem{lem}[thm]{Lemma}         

\newtheorem{bigthm}{Theorem}

\theoremstyle{definition}

\theoremstyle{definition}

\theoremstyle{definition}
\theoremstyle{remark}
\newtheorem{rem}[thm]{Remark}
\newtheorem{rems}[thm]{Remarks}

\newtheorem*{acks}{Acknowledgements}

\newtheorem*{out}{Outline}

\newtheorem*{intro-rem}{Remark}
\newtheorem*{intro-rems}{Remarks}

\begin{document}
\title{On finite domination and Poincar\'e duality}
\author{John R. Klein} 
\address{Wayne State University,
Detroit, MI 48202} 
\email{klein@math.wayne.edu} 

\begin{abstract}  The object of this paper is to show that non-homotopy finite Poincar\'e duality spaces are plentiful.
Let $\pi$ be finitely presented group. Assuming that
the reduced Grothendieck group $\tilde K_0(\Bbb Z[\pi])$ has a non-trivial
2-divisible element, we construct a
finitely dominated Poincar\'e space $X$ with fundamental group $\pi$ such that
$X$ is not homotopy finite. The dimension of $X$
can be made arbitrarily large.  
Our proof relies on a result
 which says that every finitely dominated space possesses a stable Poincar\'e duality thickening.
\end{abstract}

\maketitle
\setlength{\parindent}{15pt}
\setlength{\parskip}{1pt plus 0pt minus 1pt}
\def\smsh{\wedge}
\def\flush{\flushpar}
\def\dbslash{/\!\! /}
\def\:{\colon\!}
\def\Bbb{\mathbb}
\def\bold{\mathbf}
\def\cal{\mathcal}
\def\End{\text{\rm End}}
\def\stableto{\mapstochar \!\!\to}
\def\lr{{\ell r}}
\def\ad{\text{\rm{ad}}}
\def\tr{\sigma}
\def\Top{\text{\rm Top}}

\setcounter{tocdepth}{1}
\tableofcontents
\addcontentsline{file}{sec_unit}{entry}

\section{Introduction} \label{sec:intro}
For a group $\pi$, let $K_0(\Bbb Z[\pi])$ be the Grothendieck group of the category of finitely generated projective (left) 
$\Bbb Z[\pi]$-modules.
According to \cite{Wall_finiteness1}, a connected finitely dominated  space $X$ with fundamental group $\pi$ determines an element
\[
w(X)\in K_0(\Bbb Z[\pi])\, ,
\] 
called the {\it Wall finiteness obstruction}.
If we assume that $\pi$ is finitely presented, then $X$ has the homotopy type of a finite CW complex
if and only if $\tilde w(X)= 0$, where $\tilde w(X) \in \tilde K_0(\Bbb Z[\pi])$ is the image $w(X)$ in
the reduced Grothendieck group. 

Recall that $K_0(\Bbb Z[\pi])$ comes equipped with an involution $
\sigma$ that is
induced by mapping a finitely generated projective $\Bbb Z[\pi]$-module $P$ to its linear dual
$P^* := \hom_{\Bbb Z[\pi]}(P,\Bbb Z[\pi])$, where the latter is converted into a left module using
the involution of the group ring induced by mapping a group element to its inverse.

The main result of this paper is the following.

\begin{bigthm} \label{bigthm:ample} Let $\pi$ be a finitely presented group. Assume that $ \tilde K_0(\Bbb Z[\pi])$
has an element $\kappa$ such that $2\kappa \ne 0$.
Then there is a connected, finitely dominated Poincar\'e duality space 
$X$ such that
\begin{itemize}
\item $X$ is not homotopy finite, 
\item  $\pi_1(X) = \pi$, and
\item  $\tilde w(X) = \kappa + (-1)^d\sigma(\kappa)$, where $d = \dim X$
may be chosen to be arbitrarily large.
\end{itemize}
\end{bigthm}

\begin{rems}  (1).  Our method of proof 
shows that for any such $\kappa$, there is a positive integer $D$ such that there is an $X$ satisfying Theorem \ref{bigthm:ample} 
for every even $d >D$ or every odd $d>D$, but not both. 
\smallskip

\noindent (2). Lizhen Qin pointed out to me that Theorem \ref{bigthm:ample} is similar to Wall's \cite[thm.~1.5]{Wall_Poincare}. However, the approaches 
are quite different. Wall's method is to start with a finite complex $K$ with fundamental group $\pi$. He then
thickens $K$ to a compact $d$-manifold with boundary $(N,\partial N)$. Lastly, Wall  modifies $\partial N$
by attaching certain cells while at the same time taking care to retain Poincar\'e duality. 
The resulting Poincare space is designed so as to have finiteness obstruction $\kappa + (-1)^d\sigma(\kappa)$. 

On the other hand, our approach only requires Theorem \ref{bigthm:KL} below, which asserts that any finitely dominated space $L$ 
admits a {\it Poincar\'e thickening}, i.e.,
there is a Poincar\'e pair $(K,\partial K)$ with $K$ having the homotopy type of $L$. This is proved with the help of the {\it dualizing spectrum} of
\cite{Klein_dualizing}.
\smallskip

\noindent (3). Let $\pi = \Bbb Z/p\Bbb Z$ denote a cyclic group of odd prime order. 
The abelian group $\tilde K_0(\Bbb Z[\pi])$ has a nontrivial 2-divisible element whenever
the relative class number $h_1(p)$ is not a power of $2$ (cf. \cite[pp.~30-31]{Milnor_K-theory}).
The smallest prime satisfying this condition is $p = 23$ (cf.~\cite{wiki:Cyclotomic_field} for many other examples). Moreover,
 $h_1(p)$ is not a power of 2 whenever $p$  is an irregular prime, so there are infinitely many such $p$ (recall that an prime $p$ is irregular
 if and only if it divides  the class number $h(p)$ of the cyclotomic field generated by $e^{2\pi i/p}$).

Making use of the relative class number, Wall constructed
 a non-homotopy finite Poincar\'e duality space $X$ of dimension 4 
  \cite[cor.~5.4.2]{Wall_finiteness1}, where $\pi_1(X) = \pi = \Bbb Z/p \Bbb Z$. 
  Wall's approach is number theoretic.
 \smallskip
  
\noindent (4). A geometric framework to constructing non-homotopy finite Poincar\'e duality spaces
 is outlined in the work of Pedersen and Ranicki
\cite{Pedersen-Ranicki}, who use Siebenmann's theory of tame ends. However, in the latter paper
no examples are provided.
\end{rems}  

\begin{rem} I am not aware of a general criterion for determining when 
$\tilde K_0(\Bbb Z[\pi])$ has non-trivial 2-divisible elements. However,
if $\pi$ has a factor $H$ such that $\tilde K_0(\Bbb Z[H])$ has non-trivial 
2-divisible elements, then so will $\tilde K_0(\Bbb Z[\pi])$.  \end{rem}

The proof of Theorem \ref{bigthm:ample} will rely on another result which may be of independent interest:

 \begin{bigthm}[Stable Poincar\'e Thickening] \label{bigthm:KL} Given a finitely dominated space $L$, there is a Poincar\'e pair 
 \[
 (K,\partial K)
 \]
such that 
\begin{itemize}
\item $\partial K$ connected and finitely dominated,
\item $K$ homotopy equivalent to $L$, and  
\item  $\pi_1(\partial K) \to \pi_1(K)$ is an isomorphism.
\end{itemize}
 \end{bigthm}

\begin{rem} The proof of Theorem \ref{bigthm:KL}  is relatively easy when $L$ is a finite complex, since one may appeal to  induction
on the number of cells to embed $L$ up to homotopy in a high dimensional Euclidean space. This results in a smooth compact manifold
thickening $M$ of $L$ as in \cite{Wall_thickening}. Then $(M,\partial M)$ will fulfill the assertion of Theorem \ref{bigthm:KL}.

However, in the finitely dominated case, the inductive approach is no longer available as the number of cells of $L$ may be infinite.
Instead, our approach in this case will be homotopy theoretic: we will follow the proof
of `$3\Rightarrow 1$' of \cite[thm.~A]{Klein_dualizing}. 
\end{rem}

\begin{rem} Finitely dominated Poincar\'e duality spaces fall under the rubric of projective surgery theory \cite{Pedersen-Ranicki}.
If $(f,b)\: M \to X$ is a normal map from a compact $d$-manifold $M$ to a finitely dominated Poincar\'e duality space $X$ of dimension $d$, then
it has a projective surgery obstruction $\sigma^p(f,b) \in L^p_d(\pi_1(X))$. 
If $d \ge 4$, then $(f,b) \times 1\: M \times S^1 \to X\times S^1$ will be
normally cobordant to a homotopy equivalence if and only if $\sigma^p(f,b) = 0$.

\end{rem}

\begin{out}  \S\ref{sec:prelim}  is mostly language.
In \S\ref{sec:thm_ample} we show how Theorem \ref{bigthm:KL} implies Theorem \ref{bigthm:ample}. In
\S\ref{sec:KL} we prove Theorem \ref{bigthm:KL}.
\end{out}

\begin{acks} I am indebted to Ian Leary for drawing my attention to
Wall's example \cite[cor.~5.4.2]{Wall_Poincare}. I am also grateful to Lizhen Qin for
 drawing my attention to Wall's \cite[thm.~1.5]{Wall_Poincare}, which I had either forgotten or was unaware of.

This research was supported by the U.S.~Department of Energy, Office of Science, under Award Number DE-SC-SC0022134.
\end{acks}

\section{Conventions} \label{sec:prelim}

\subsection{Spaces} Let $\Top$ be the Quillen model category of compactly generated weak Hausdorff spaces
\cite{Hirschhorn_model}. The
weak equivalences of $\Top$ are the weak homotopy equivalences, and the fibrations are the Serre fibrations. The cofibrations 
are defined using the right lifting property with respect to the trivial fibrations.
In particular, every object $\Top$ is fibrant. An object is cofibrant whenever it is a retract of a cell complex. We let $\Top_\ast$
denote the category of based spaces. Then $\Top_\ast$ inherits a Quillen model structure from $\Top$ by means of the forgetful functor
$\Top_\ast\to \Top$.

An object $\Top$ or $\Top_\ast$ is {\it finite} if it is a finite cell complex. It is {\it homotopy finite}
if it weakly equivalent to a finite object. A object is $X$ is {\it finitely dominated} if it is a retract of a 
homotopy finite object.

If $X$ is an unbased space, we write $X_+$ for the based space $X \amalg \ast$ given by taking
the disjoint union with a basepoint.

\subsection{Poincar\'e duality spaces} Recall that an object $X \in \Top$ is {\it Poincar\'e duality space} of (formal) dimension $d$ if
there exists a pair
\[
(\cal L,[X])
\]
in which $\cal L$ is a rank one local coefficient system and 
\[
[X] \in H_d(X;\cal L)
\]
is a {\it fundamental class} such that for all local coefficient systems $\cal B$, the cap product homomorphism
\[
\cap [X] \: H^\ast(X;\cal B) \to H_{d-\ast}(X;\cal L\otimes\cal B)
\]
is an isomorphism in all degrees (cf.~\cite{Wall_Poincare}, \cite{Klein_Poincare}).  
The Poincar\'e
spaces considered in this paper cofibrant.
If the pair $(\cal L, [X])$ exists it is determined up to unique isomorphism. 
Also note that closed manifolds
are homotopy finite Poincar\'e duality spaces.

\begin{rem}  If $X$ is a connected Poincar\'e space, then $X$ is finitely dominated if and only
if $\pi_1(X)$ is finitely presented (cf. \cite[thm~D]{Klein_dualizing}).
\end{rem}

\section{The proof Theorem \ref{bigthm:ample}} \label{sec:thm_ample}

In this section we show how Theorem \ref{bigthm:KL} implies Theorem \ref{bigthm:ample}. The proof
of Theorem \ref{bigthm:KL} will appear in the next section.

As mentioned in the introduction,  the group ring $\Bbb Z[\pi]$ comes equipped with a canonical involution which on group elements is defined
by $g\mapsto g^{-1}$. One may use this involution to convert right modules to left modules and {\it vice versa}.
The Grothendieck group $K_0(\Bbb Z[\pi])$ is, in turn, equipped with an involution that is induced by the operation 
$P \mapsto P^\ast$, in which $P$ is a finitely generated projective left $\Bbb Z[\pi]$-module and  $P^* = \hom_{\Bbb Z[\pi]}(P,\Bbb Z[\pi])$
is its linear dual. The latter is a finitely generated projective right $\Bbb Z[\pi]$-module which we identify as a finitely generated projective right 
$\Bbb Z[\pi]$-module.  Denote the involution on $K_0(\Bbb Z[\pi])$ by $\sigma$. Then $\sigma$ restricts to an involution on the reduced
Grothendieck group $\tilde K_0(\Bbb Z[\pi])$.

 Let $(K,\partial K)$ be a finitely dominated Poincar\'e pair of dimension $d$.  Assume that both $K$ and $\partial K$ are connected
 and $\pi_1(\partial K) \to \pi_1(K) = \pi$ is an isomorphism. Then we have
 
 \begin{lem} [Wall {\cite[thm.~1.4]{Wall_finiteness1}}] Let $\kappa = \tilde w(K)$.  Then
 \[
 \tilde w(\partial K) = \kappa + (-1)^{d-1} \sigma(\kappa)\, .
 \]
 \end{lem}

 With $(K,\partial K)$ as above, the double
 \[
 X := K \cup_{\partial K} K
 \]
 is a finitely dominated Poincar\'e duality space of dimension $d$. Assume that $\partial K$ is connected and homotopy finite.
Then by additivity of the finiteness obstruction, we infer that 
 \[
 \tilde w(X) = \tilde w(K) + \tilde w(K) - \tilde w(\partial K) = \kappa + (-1)^{d}\sigma(\kappa)\, .
 \]
 
 \begin{cor} If $2\kappa \ne 0$, then either $\partial K$ is not homotopy finite or $X$ is not homotopy finite.
 \end{cor}
 
  \begin{proof} If $\partial K$ is homotopy finite, then $ \kappa + (-1)^{d-1} \sigma(\kappa)=0$. If $X$ is homotopy finite, then
 $ \kappa + (-1)^{d} \sigma(\kappa)=0$. Consequently, if both $\partial K$ and $X$ are homotopy finite, we
 infer that $ \sigma(\kappa) = -\sigma(\kappa)$ and this implies that $\kappa = -\kappa$ or
 $2\kappa = 0$, a contradiction.
 \end{proof}

Recall that for a map  of spaces $Y \to B$, the associated $j$-fold unreduced {\it fiberwise suspension}  is
\[
S^j_B Y := Y \times D^j \cup_{Y\times S^{j-1}}  B \times S^{j-1}\, .
\]
Applying this construction to $\partial K \to K$, we obtain a finitely dominated Poincar\'e duality space
\[
X_j := S^j_K \partial K
\]
having formal dimension $d_{j-1} := d + j-1$. Note that $X_0 = \partial K$ and $X_1$ is weakly equivalent to the double of $(K,\partial K)$.

\begin{cor} \label{cor:sequence} Assume $2\kappa \ne 0$.  Then for every single $j \ge 0$, at least one of
the Poincar\'e duality spaces $X_{2j},X_{2j+1}$ is not homotopy finite. 
Furthermore, $\tilde w(X_j) = \kappa + (-1)^{d_{j-1}} \sigma(\kappa)$.
\end{cor}

\begin{proof}[Proof of Theorem \ref{bigthm:ample}]
 Let $\kappa \in \tilde K_0(\Bbb Z[\pi])$ be  such that $2\kappa \ne 0$.  Then by \cite[thm.~F]{Wall_finiteness1} there is 
 a finitely dominated space $L$ with fundamental group $\pi$ such that $\tilde w(L) = \kappa$.
By Theorem \ref{bigthm:KL}, there is a finitely dominated Poincar\'e pair $(K,\partial K)$
 and a homotopy equivalence $K\simeq L$.  
 The proof is completed by applying Corollary \ref{cor:sequence}.
 \end{proof}
 
 \section{The proof of Theorem \ref{bigthm:KL}} \label{sec:KL}
 
 As mentioned in the introduction, we will follow the proof of  `$3 \Rightarrow 1$' of \cite[thm.~A]{Klein_dualizing}. 
 However, some details in that proof were omitted in the finitely dominated case.  We provide those details here.

We first review the various notions of finiteness in the equivariant setting. 
Let $G$ be a topological group object of $\Top$ whose underlying space is cofibrant. Let $\Top_\ast(G)$
denote the category of based $G$-spaces.
An object 
$Z\in \Top_\ast(G)$ is  {\it finite} if it is built up from a point by finitely many free $G$-cell attachments, where by a free $G$-cell,
we mean $D^k \times G$.  Similarly, $Z$ is 
{\it homotopy finite} if it is  weakly equivalent to a $G$-finite object. Lastly, $Z$ is {\it finitely dominated} if it is
a retract of a homotopy finite object. One also has the corresponding notions of finiteness in the category of 
unbased $G$-spaces as well as in the category of 
(naive) $G$-spectra.
 
 \begin{proof}[Proof of Theorem \ref{bigthm:KL}]
 We may assume without loss in generality that $L = BG$ for a suitable cofibrant topological group $G$.
 The idea is to construct a finitely dominated based $G$-space $Z$ and a $G$-equivariant weak equivalence
 \[
 \Sigma^\infty Z \simeq \Sigma^n \cal D_G\, ,
 \]
 where $n \ge 0$ is some integer and $\cal D_G = S[G]^{hG}$ is the {\it dualizing spectrum} of $G$ (i.e., the homotopy fixed points
 of $G$ acting on $S[G]$, where the latter denotes the suspension spectrum of $G_+$).

 Assuming this to be the case, then as in the proof of `$3 \Rightarrow 1$' of \cite[thm.~A]{Klein_dualizing} 
 the desired Poincar\'e pair $(K,\partial K)$ is given by the pair
 of unreduced Borel constructions
 \[
 (EG\times_G CZ,EG\times_G Z)\, ,
 \]
 where $CZ$ is the cone on $Z$. 
 
 We now proceed with the construction of $Z$.
 Choose a factorization of the identity map
 \[
 L \to X \to L
 \]
 in which the unbased space $X$ is homotopy finite. 
 Let $\tilde L = EG$ and  $\tilde X = X\times^L EG$  be the fiber product of $X$ and $EG$ along $L$. Then one has a factorization of $G$-spaces
 \[
 \tilde L \to \tilde X \to \tilde L
 \]
 in which $\tilde X$ is a homotopy finite unbased $G$-space.
 
 Then the equivariant dual
 \[
 \hom_G(\tilde X_+,S[G])
 \]
 is a homotopy finite $G$-spectrum. This means  that there is an $n >0$, a homotopy finite based 
 $G$-space $Y$ and an equivariant weak equivalence
 \[
 \Sigma^\infty Y \simeq \Sigma^n \hom_G(\tilde X_+,S[G])\, .
 \]
 Moreover, the dualizing spectrum $\cal D_G = \hom_G(\tilde L_+,S[G])$ is an equivariant retract of $\hom_G(\tilde X_+,S[G])$. It follows that
 $\Sigma^n \cal D_G$ is an equivariant homotopy retract of the homotopy finite $G$-spectrum $\Sigma^\infty Y$. 
 Consequently, there is a factorization
 \[
 \Sigma^n \cal D_G @>s' >> \Sigma^\infty Y @> r' >>  \Sigma^n \cal D_G\, .
 \]
 in which $r'\circ s'$ is equivariantly homotopic to the identity. Consider the self-map $s'\circ r'\:  \Sigma^\infty Y \to  \Sigma^\infty Y$.
 As $Y$ is homotopy finite, there is an integer $k>0$ such that $s'\circ r'$ equivariantly desuspends to a $G$-map $\Sigma^k Y \to \Sigma^k Y$.
 Let $Z$ denote the mapping telescope of 
 \[
\Sigma^k Y@> s'\circ r' >> \Sigma^k Y @> s'\circ r' >> \Sigma^k Y @> s'\circ r' >> \cdots
 \]
Then the based $G$-space $Z$ is an equivariant retract of $\Sigma^k Y$ and  $\Sigma^\infty Z \simeq \Sigma^n \cal D_G$, as was to be
 proved.
   \end{proof}

\bibliography{john}


\end{document}